\documentclass[12pt]{amsart}
\usepackage[colorlinks=true, pdfstartview=FitV, linkcolor=blue, citecolor=blue, urlcolor=blue]{hyperref}

\usepackage{amssymb,amsmath, amscd, amsfonts}
\usepackage{times, verbatim}
\usepackage{graphicx}
\usepackage[english]{babel}
 
\usepackage{nicematrix}
\usepackage{enumerate}
\usepackage{ulem}
\usepackage{anysize}
\marginsize{3cm}{3cm}{3cm}{3cm}

\usepackage[all]{xy}
\input xy
\xyoption{all}

\usepackage{mathtools}
\usepackage{leftidx} % g^t, but t other side
\usepackage{lipsum} % Only for filler text

\theoremstyle{plain}
\newtheorem{theorem}{Theorem}[section]

\newtheorem{lemma}[theorem]{Lemma}

\newtheorem{corollary}[theorem]{Corollary}

\theoremstyle{definition} \theoremstyle{definition}
\newtheorem{remark}[theorem]{Remark}
\newtheorem{example}[theorem]{Example}

\theoremstyle{remark}

\newcommand{\Z}{\mathbb{Z}}
\newcommand{\C}{\mathbb{C}}

\newcommand{\Hom}{{\rm Hom}}

\def\G{{\rm G}}

\def\Sp{{\rm Sp}}

\def\GL{{\rm GL}}

\def\SO{{\rm SO}}

\def\Sym{{\rm Sym}}

\newcommand{\pilam}{\Pi_{\lambda}}

\newcommand{\pimu}{\Pi_{\mu}}

\begin{document}

\title[New Perspective of Pieri Rule]
{Pieri Rule for Classical Groups, A New Perspective}

\begin{abstract}
We study a new perspective on a certain Pieri rules for classical groups. Furthermore, we extend a fundamental theorem of Kostant concerning tensor products for classical groups. We show that a certain form of the Pieri rule is equivalent to the converse of this extended version of Kostant’s theorem. In addition, we show an equivalence between the Pieri rule and the branching rule for general linear groups. 
\end{abstract}

\author{Dibyendu Biswas}

\address{
	Chennai Mathematical Institute, Chennai, India, 603103}
%\address{Indian Institute of Technology Bombay, Powai, Mumbai, India-400076}	

\email{dibubis@gmail.com, dibyendub@cmi.ac.in}


\maketitle
    {\hfill \today}
    
\tableofcontents
\section{Introduction}

The Pieri rule describes the decomposition of tensor products of irreducible representations with symmetric or exterior power of defining representations. It is originally formulated in the context of symmetric functions and the representation theory of the general linear groups.
 For general linear groups (in type $A$), it states that the product of a Schur function $s_\lambda$ with a complete (or elementary) symmetric function equals as a sum of Schur functions indexed by partitions obtained by adding a horizontal (or vertical) strip. This rule underlies the Littlewood--Richardson theory and has implications in algebraic combinatorics, representation theory, and geometry.
For classical groups of types $B$, $C$, and $D$, Pieri-type rules describes the decomposition of tensor products of irreducible representations with symmetric (exterior) power of defining representations, by removing (adding) and adding (removing) certain horizontal (vertical) strips. See an explicit version for the case of exterior powers in Theorem~\ref{thm pieri dp2}.

Pieri's rule for the classical groups was established by Sundaram. She provided a combinatorial proof for the symplectic group in \cite[Theorem 4.1]{Sunsp} and for the orthogonal groups in \cite[Theorem 5.3]{Sunortho}. Soichi Okada also gave proof of Pieri's rule using  the universal character ring and specialization homomorphism in \cite{SO}.

The Pieri rule has deep connections across several areas of mathematics. Geometrically, it describes the multiplication of Schubert classes on the Grassmannian, and was reformulated algebraically by Chevalley \cite{Chevalley1951} within the cohomology ring of flag varieties as a special case of his formula for the product of a Schubert class with a Schubert divisor. Subsequent developments by Bott and Samelson \cite{BottSamelson1958} introduced a geometric framework through their study of Bott--Samelson resolutions, and later work by Buch, Kresch, and Tamvakis \cite{BuchKreschTamvakis2006} extended Pieri-type formulas to isotropic Grassmannians of types $B, C$ and $D$.

From the representation-theoretic viewpoint, the Pieri rule governs the decomposition of the tensor product of a highest weight representation with a fundamental representation. This perspective fits naturally within the broader frameworks of Littelmann’s path model \cite{Littelmann1994, Littelmann1995} and Kashiwara’s crystal basis theory \cite{Kashiwara1990, Kashiwara1995}, which provide uniform combinatorial and geometric approaches to tensor product decompositions across all types, and even for Kac--Moody algebras.

All results in this paper are over $\C$, and all groups considered are connected and reductive. We study a new perspective on a certain Pieri rules (in Theorem~\ref{thm pieri dp2}) for the classical groups. Our proof of Theorem~\ref{thm pieri dp2} arises from a more general result, Theorem~\ref{thm pieri dp}, which extends a certain form of the Pieri rule (Theorem~\ref{thm ext dp}) for the general linear groups to arbitrary connected reductive groups.
Theorem~\ref{thm pieri dp} gives an explicit description of the tensor product $\Pi_{\lambda} \otimes U$ whenever 
$\lambda$ lies sufficiently deep in the dominant chamber, precisely  $\lambda + \rho + \mu$ is a dominant weight for all weights $\mu$ in the representation $U$.  As it reduces tensor product problems to the weight structure of $U$ and provides a uniform method for all connected reductive groups 
for controlling highest weights of irreducible constituents appeared in tensor product in various contexts, including translation principles 
\cite{Kostant1963,Littelmann1995} and the study of tensor product asymptotics \cite{Heckman1982}. Here the translation principle refers to the method of relating representations with different highest weights by shifting them across walls of the Weyl chambers, while asymptotic questions concern the behavior of tensor product multiplicities under scaling of highest weights.

Furthermore, we extend a fundamental theorem of Kostant on tensor product decompositions (Theorem~\ref{thm kostant initial}) of classical groups and obtain a more general theorem, Theorem~\ref{thm gbt parabolic}, which we refer to as the extended Kostant theorem. Theorem~\ref{thm gbt parabolic} provides a parabolic generalization of 
Kostant's classical constraint on tensor product multiplicities. 
It shows that any irreducible representation of 
$G$ occurring in $\Pi_{\lambda} \otimes \Pi_{\mu}$
must already arise, in a precise sense, within the tensor product of the corresponding 
$M$-modules, where $M$
is the Levi factor of a parabolic subgroup.
We use this theorem to make an observation on representations of classical groups.  
  We show, in Theorem~\ref{thm kostant parabolic} and Remark~\ref{rem equivalent to CKT}, that a certain form of the Pieri rule is in fact equivalent to the converse of this extended Kostant theorem for symplectic groups.
 
 In Subsection~\ref{subsec equiv pieri branch}, we establish an equivalence between the Pieri rule and the branching rule for general linear groups. This equivalence is formulated explicitly in equation~\eqref{eqn: equivalance equation}. We show that this relationship can be naturally interpreted as a reciprocity phenomenon, which is stated in equation~\eqref{eqn: reciprocity}.
 We prove this equivalence using a simplest case of Howe duality, precisely Corollary~\ref{cor Howe}, using the natural action of $\GL(n, \C) \times \GL(m, \C)$, $n \leq m$, on the space of polynomial functions on the space of $n \times m$ matrices.

The main results of this paper are Theorem~\ref{thm pieri dp2}, Theorem~\ref{thm gbt parabolic}, Theorem~\ref{thm kostant parabolic}.
An equivalence between the Pieri rule and the branching rule for the general linear groups in subsection~\ref{subsec equiv pieri branch}, is another significant result of the paper. 
	We mention that a version of Theorem~\ref{thm pieri dp} is known in the literature, it does not occur in the precise form that we present here. 
	Theorem~\ref{thm pieri dp} is significant in that it provides a particular instance in which the converse of Kostant’s theorem holds strongly, as noted in Remark~\ref{rem strong ckt}. See also part (2) of Remark~\ref{rem ckt in pieri}. At present, only this special cases of the converse of Kostant are known, and a general characterization of when this converse holds remains an open problem.

In particular, the extended Kostant Theorem~\ref{thm gbt parabolic} suggests that Pieri-type rules, branching phenomena, and known cases of the converse to Kostant’s theorem may be understood from a common Levi-theoretic viewpoint.
 It shows how Levi subgroups govern tensor product multiplicities. We hope that this perspective will lead to further connections with recent work on eigencones, Horn-type inequalities, and parabolic refinements, and other ongoing developments in the subject \cite{Atiyah1982,Klyachko1998,KapovichMillson2005,BelkaleKumar2006,sagar}.

\section{Preliminaries}
A partition is a weakly decreasing sequence
$
\lambda=(\lambda_1,\lambda_2,\ldots)
$
of nonnegative integers with finite sum, i.e.,
$
\sum_{i\ge 1}\lambda_i<\infty.
$
To a partition $\lambda$ one associates its Young (Ferrers) diagram, in which the $i$-th row has $\lambda_i$ left-justified boxes. The conjugate partition $\lambda^t=(\lambda_1',\ldots,\lambda_r')$ is obtained by interchanging rows and columns of the Young diagram (equivalently, by reflection across the $45^\circ$ line).
Let $\lambda=(\lambda_i)_{i\ge 1}$ be a partition.
The length $l(\lambda)$ is the largest index $s$ such that $\lambda_s\neq 0$.
The size $|\lambda|$ is
$
|\lambda|=\sum_{i\ge 1}\lambda_i \;=\; \sum_{i=1}^{l(\lambda)}\lambda_i.
$
Let $\mu$ and $\lambda$ be partitions with $\mu\subset \lambda$ (i.e.\ $\mu_i\le \lambda_i$ for all $i$).
The skew diagram $\lambda/\mu$ is the set-theoretic difference of the Young diagrams of $\lambda$ and $\mu$.
Its size is $|\lambda/\mu|=|\lambda|-|\mu|$.
The skew diagram $\lambda/\mu$ is a horizontal (vertical) $r$-strip if it contains at most one box in each column (row) and $|\lambda/\mu|=r$.
The skew diagram $\lambda/\mu$ is a horizontal strip if and only if $
\lambda_1 \ge \mu_1 \ge \lambda_2 \ge \mu_2 \ge \cdots$. Note that $\lambda/\mu$ is a horizontal strip if and only if  $\lambda^t/\mu^t$ is a vertical strip.

Let $G$ be a connected reductive algebraic group over $\C$, with Borel subgroup $B\subset G$ and maximal torus $T\subset B$.
Let $\Phi$ be the root system of $(G,T)$, $\Phi^+$ a choice of positive roots, and
$\rho=\frac{1}{2}\sum_{\alpha\in \Phi^+}\alpha$.
A Borel subgroup of $G$ is a maximal connected solvable subgroup, and a parabolic subgroup is any algebraic subgroup $P$
satisfying $B\subset P \subset G$.
Finite-dimensional irreducible algebraic representations of $G$ are classified by dominant integral weights.
Let $W$ denote the Weyl group of $G$.

\begin{theorem}[Weyl Character Formula]\label{weyl char formula}
	Let $\lambda$ be a dominant integral weight, and let $\Pi_\lambda$ be the corresponding irreducible representation of $G$.
	Then
	\[
	\operatorname{ch}(\Pi_\lambda)\;
	\sum_{w\in W}\operatorname{sgn}(w)\,e^{w(\rho)}
	=
	\sum_{w\in W}\operatorname{sgn}(w)\,e^{w(\lambda+\rho)}.
	\]
\end{theorem}

 Finite-dimensional irreducible algebraic representations of $\GL(n,\C)$ are parameterized by highest weights
$\lambda_1\ge \lambda_2 \ge \cdots \ge \lambda_n$. In this paper, we consider only the polynomial representations of $\GL(n, \C)$, thus we always will impose the condition that $\lambda_n \geq 0$. Often, we are reduced to the condition $\lambda_n \geq 0$ because of twisting by the determinant character $\operatorname{det}: \GL(n, \C) \to \C^{*}$, but in any case, for our work, we prefer to assume throughout the paper $\lambda_n\geq 0$.
 Similarly, irreducible finite dimensional representations of classical groups $\Sp(2n, \C)$, $\SO(2n+1, \C)$, $\SO(2n, \C)$ are parameterized by their highest weights $\lambda_1 \geq \lambda_2 \geq \cdots \geq \lambda_n$ with $\lambda_i$ integers which are $\geq 0$  $\forall i$ for $\Sp(2n, \C)$ and $\SO(2n+1, \C)$, but for $\SO(2n, \C)$, it may happen that $\lambda_n \leq 0$; for $\SO(2n, \C)$, $\lambda_{n-1} \geq |\lambda_n|$. We will often use the well-known fact about classical groups that any finite-dimensional(algebraic) representation of them is completely reducible.
 
 Let $\varepsilon_1,\ldots,\varepsilon_n$ denote the standard orthonormal basis of the weight lattice.  
 For the four classical families we have:
\begin{equation}\label{rho}
	\rho= \begin{cases}
		\sum_{i=1}^{n} (n-i) \varepsilon_i, & \text{ for } \quad \GL(n, \C) ,  \SO(2n, \C), \\ 
		\sum_{i=1}^{n}	(n-i+1)\varepsilon_i, & \text{ for } \quad \Sp(2n, \C), \\
		\sum_{i=1}^{n}	(n-i+\frac{1}{2})\varepsilon_i, & \text{ for } \quad \SO(2n+1, \C).
	\end{cases} 
\end{equation}

Now, let
$
G = \Sp({2n},\mathbb{C}), \SO({2n+1},\mathbb{C}),  \text{or }  \SO({2n},\mathbb{C}).
$
Maximal parabolic subgroups correspond to removing a single simple root from the Dynkin diagram of $G$.  
Equivalently, they are  precisely the stabilizers of isotropic subspaces of dimension $k$, where $1 \le k \le n$. 
Up to $G$-conjugacy, there is exactly one conjugacy class of such maximal parabolic subgroups for each $k$ with $1 \le k \le n$, except in type $D_n$, where the case $k = n$ splits into two non-conjugate classes.
Each maximal parabolic subgroup admits a Levi decomposition. Levi factor for $k \neq n$ case is given by :
\[
L \;\cong\;
\begin{cases}
	\GL(k,\C)\times \Sp({2(n-k)},\C), & G=\Sp({2n}, \C),\\[4pt]
	\GL(k,\C)\times \SO({2(n-k)+1}, \C), & G=\SO({2n+1}, \C),\\[4pt]
\GL(k,\C)\times \SO({2(n-k)}, \C), & G=\SO({2n}, \C).
\end{cases}
\]
When $k = n$ (Siegel parabolic case), the Levi factor is $\GL(n,\mathbb{C})$ for all three families of groups.
In type $D_n$, for $k = n$, there are two non-conjugate maximal parabolic subgroups, each with Levi factor $\GL(n,\mathbb{C})$.

\section{Pieri Rule for General Linear Groups}\label{sec pieri gen dp}

The aim of this section is to prove the both versions of Pieri's rule for general linear groups. Recall that Pieri's rule are of two kind, one in which for an irreducible highest weight module $\pilam$ of $\GL(V)$ or more generally a classical group $G(V)$ defined using a quadratic or symplectic form on $V$, one calculates $\pilam \otimes \Lambda^i V$, and the other form of the Pieri rule calculates $\pilam \otimes \Sym^i V$. The proof outlined below first proves the Pieri's rule for $\pilam \otimes \Lambda^i V$, and then deduces it for
$\pilam \otimes \Sym^i V$. We have simply elaborated the proof given by Macdonald \cite[(I.5.17)]{IGM}. 

\begin{theorem}\label{thm ext dp}
	Let $V$ be an $n$-dimensional vector space over $\C$, identified to $\C^n$. Then for a finite dimensional highest weight module $\pilam$ of $\GL(V)$,
	$$\pilam \otimes \Lambda^i V = \sum_{\mu} \pimu,$$
	a multiplicity one sum of irreducible highest weight modules, with $\mu=(\mu_1 \geq \cdots \geq \mu_n \geq 0)$ which is obtained by adding $i$ boxes to the Young diagram of $\lambda=(\lambda_1 \geq \cdots \geq \lambda_n \geq 0)$ such that no two boxes are added in the same row.
\end{theorem}

\begin{proof}
	The proof of this theorem is a direct consequence of the Weyl Character formula according to which $\Theta_{\lambda}(t)$, the character of $\pilam$ at the diagonal matrix $t=(t_1, \ldots, t_n)$ is given by
	\begin{align*}
		\Theta_{\lambda}(t)=\frac{\sum_{w \in W} (-1)^w t^{w(\lambda + \rho)}}{\sum_{w \in W} (-1)^w t^{w \rho}}.
	\end{align*}
	Let $$\Delta(t)=\sum_{w} (-1)^w t^{w \rho}.$$
	All the weights of the irreducible representation $\Lambda^iV$ are of the form 
	$$e_{j_1} + e_{j_2} + \cdots + e_{j_i}, \quad \text{ with } 1 \leq j_1 < j_2 < \cdots j_i \leq n.   $$ 
	Notice that all the weights of the representation $\Lambda^i V$ of $\GL(V)$ are in the $W$-orbit of the highest weight $\omega_i=e_1 + \cdots + e_i$, with $\omega_i(t)=t_1\cdots t_i$.
	Therefore,
	\begin{align*}
		\Delta(t)	\Theta_{\lambda}(t) \Theta(\Lambda^i V)(t)=\frac{1}{|W_i|} \sum_{w \in W} (-1)^w t^{w(\lambda + \rho)} \sum_{w \in W}  t^{w(\omega_i)},
	\end{align*}
	where $W_i=$ stabilizer of $\omega_i$ in $W$.
	
	Therefore,
	\begin{align*}
		\Delta(t)	\Theta_{\lambda}(t) \Theta(\Lambda^i V)(t)=\frac{1}{|W_i|} \sum_{w, w^{'} \in W } (-1)^w t^{w(\lambda + \rho + w^{'}(\omega_i))}.
	\end{align*}
	
	Now, note the following two observations:
	\begin{enumerate}
		\item[(i)] The weight $\lambda + \rho$ is strictly dominant, i.e., $\langle \lambda + \rho, \alpha^{\vee}_a \rangle >0$ for all coroots $\alpha^{\vee}_a$, but $\langle \lambda + \rho + w(\omega_i), \alpha^{\vee}_a \rangle \geq 0$ for all coroots $\alpha^{\vee}_a$.
		In fact, if $\lambda + \rho + w(\omega_i)$ is  dominant but not strictly dominant, then it is stabilized by a simple reflection in the Weyl group $W$ and hence
		$$\sum_{w \in W } (-1)^w t^{w(\lambda + \rho + w^{'}(\omega_i))}=0.$$
		
		\item[(ii)] The weight $\lambda + \rho + w^{'}(\omega_i)$ which is dominant by (i), is strictly dominant $\iff$ $\lambda + \rho + w^{'}(\omega_i)=\mu_{w^{'}} + \rho$ such that $\mu_{w^{'}}$ is dominant and is obtained by adding $i$ boxes in the Young diagram of $\lambda$ such that no two boxes are added in the same row.
	\end{enumerate}
	
	By observation (i) and (ii) above,
	\begin{align*}
		\Delta(t)	\Theta_{\lambda}(t) \Theta(\Lambda^i V)(t)&=\frac{1}{|W_i|} \sum_{w, w^{'} \in W } (-1)^w t^{w(\lambda + \rho + w^{'}(\omega_i))},\\
		&=\frac{1}{|W_i|} \sum_{w, w^{'} \in W } (-1)^w t^{w(\mu_{w^{'}} + \rho)},
	\end{align*}
	where in the second sum $w^{'}$  varies over those elements of the Weyl group $W$ such that $\lambda + \rho + w^{'}(\omega_i)=\mu_{w^{'}} + \rho$ is strictly dominant. Observe that the set of such $w^{'} \in W$ is a union of coset spaces $W/W_i$ in which each coset contributes to the same $\mu_{w^{'}}$.
	
	Thus,
	\begin{align*}
		\Delta(t)	\Theta_{\lambda}(t) \Theta(\Lambda^i V)(t)&=\frac{1}{|W_i|} \sum_{w, w^{'} \in W } (-1)^w t^{w(\lambda + \rho + w^{'}(\omega_i))},\\
		&= \sum_{w \in W,  w^{'} \in W/W_i } (-1)^w t^{w(\mu_{w^{'}} + \rho)}, \\
		&=\sum_{w^{'} \in W/W_i } \Theta_{\mu_{w^{'}}} \Delta(t) .
	\end{align*} 
	
	Therefore,
	\begin{align*}
		\Theta_{\lambda}(t) \Theta(\Lambda^i V)(t)&=\sum_{w^{'} \in W/W_i } \Theta_{\mu_{w^{'}}},
	\end{align*}
	where $w^{'} \in W/W_i$ is such that $\lambda + \rho + w^{'}(\omega_i)=\mu_{w^{'}} + \rho$ is strictly dominant proving the theorem by the observation (ii) above.
\end{proof}

Using Theorem \ref{thm ext dp} for $\GL(4, \C)$, we calculate:
$\Pi_{(3,2,1,0)} \otimes \Lambda^2(\C^4) = \Pi_{(3,2,2,1)} + \Pi_{(3,3,1,1)} + \Pi_{(3,3,2,0)} + \Pi_{(4,2,1,1)}+ \Pi_{(4,2,2,0)} + \Pi_{(4,3,1,0)}$.

Using Theorem \ref{thm pieri gen dp} for $\GL(4, \C)$, we calculate:
$\Pi_{(3,2,1,0)} \otimes \Sym^2(\C^4) = \Pi_{(3,2,2,1)} + \Pi_{(3,3,1,1)} + \Pi_{(3,3,2,0)} + \Pi_{(4,2,1,1)} + \Pi_{(4,2,2,0)} + \Pi_{(4,3,1,0)} + \Pi_{(5,2,1,0)}$.

\begin{theorem}[Pieri Rule]\label{thm pieri gen dp}
	Let $V$ be an $n$-dimensional vector space over $\C$, identified to $\C^n$. Then for a finite dimensional highest weight module $\pilam$ of $\GL(V)$,
	$$\pilam \otimes \Sym^i V = \sum_{\mu} \pimu,$$
	a multiplicity one sum of irreducible highest weight modules $\pimu$, with $\mu=(\mu_1 \geq \cdots \geq \mu_n \geq 0)$ which is obtained by adding $i$ boxes to the Young diagram of $\lambda=(\lambda_1 \geq \cdots \geq \lambda_n \geq 0)$ such that no two boxes are added in the same column.
\end{theorem}

\begin{remark}\label{remark eq pieri}
	An equivalent way to assert Pieri's Rule (Theorem~\ref{thm pieri gen dp}) is: 
	$\pilam \otimes \Sym^i V = \sum_{\mu} \pimu,$
	the sum over all $\mu$ with
	\begin{enumerate}
		\item $\mu_1 \geq \lambda_1 \geq \mu_2 \geq \lambda_2 \geq \cdots \geq \mu_n \geq \lambda_n \geq 0$,
		\item $\sum_{j=1}^n \mu_j=\sum_{j=1}^n \lambda_j+k$.
	\end{enumerate}
\end{remark}

\begin{proof}(of Theorem~\ref{thm pieri gen dp})
	The proof of this theorem is a direct consequence of the previous theorem about $\pilam \otimes \Lambda^i V = \sum_{\mu} \pimu$ by standard methods that we elaborate now.
	
	Let $\Lambda$ be the ring of symmetric functions in infinitely many variables $t_1, t_2, \ldots, t_n, \ldots$ over $\Z$. The ring $\Lambda$ is a polynomial algebra $\Z[e_i]$ on the space of elementary symmetric functions
	$$ e_i=\sum_{n_1< n_2 < \cdots < n_i} t_{n_1}t_{n_2} \cdots t_{n_i}.$$
	
	Let $h_i$ denotes the complete symmetric function 
	$$ h_i=\sum_{n_1\leq  n_2 \leq \cdots \leq  n_i} t_{n_1}t_{n_2} \cdots t_{n_i}.$$
	
	Then, we  have the formal identities:
	$$\prod_{j=1}^{\infty} \frac{1}{(1-t_jX)}=\sum_{i=0}^{\infty} h_i X^i,$$
	
	$$\prod_{j=1}^{\infty} {(1-t_jX)}=\sum_{i=0}^{\infty} (-1)^i e_i X^i,$$
	therefore,
	\begin{equation}\label{eh prod 1}
		\sum_{i=0}^{\infty} h_i X^i \sum_{i=0}^{\infty} (-1)^i e_i X^i =1.
	\end{equation}
	
	As $\Lambda=\Z[e_1, e_2, \ldots]$, a polynomial ring in infinitely many variables, one can define a homomorphism of rings $\omega: \Lambda \to \Lambda$ by $\omega (e_i)=h_i$. Applying  $\omega$ to the identity (\ref{eh prod 1}), we deduce that $\omega(h_i)=e_i$, thus $\omega^2=\text{ Id}$, proving also that 
	$$\Lambda=\Z[e_1, \ldots, e_n, \ldots]=\Z[h_1, \ldots, h_n, \ldots].$$
	
	Recall that for a highest weight $\lambda=(\lambda_1 \geq \cdots \geq \lambda_n \geq 0)$, we have a compatible family of symmetric functions $s_{\lambda}(m) \in \Z[t_1, \ldots, t_m]$ for all $m \geq n$, defining an element $s_{\lambda} \in \Lambda$ such that the image of $s_{\lambda}$ under the natural homomorphism $\Lambda \rightarrow \Z[t_1, \ldots, t_m]$ is  $s_{\lambda}(m)$ for all $m \geq n$.
	
	Now a basic fact \cite[I.5.6]{IGM}
	%(Lemma~\ref{lem involution})
	 about the involution $\omega$ on $\Lambda$ is that 
	$$ \omega(s_{\lambda})=s_{\lambda^t},$$ 
	where $s_{\lambda^t}$ is the Schur function in infinitely many variables associated to the Young diagram which is the transpose of the Young diagram of $\lambda$,
	generalizing
	$$\omega(e_i)=h_i. $$
	
	Now $\omega$ being a ring homomorphism, Pieri's theorem about $s_{\lambda} \cdot e_i$:
	$$s_{\lambda} \cdot e_i= \sum s_{\mu}, $$ 
	after applying $\omega$ looks like:
	$$w(s_{\lambda}) \cdot w(e_i)= \sum w(s_{\mu}), $$ 
	i.e., 
	$$s_{\lambda^t} \cdot h_i= \sum s_{\mu^t}, $$ 
	which is the desired version of the Pieri's theorem in our context. Since we are proving the Pieri's rule  for $\GL(n, \C)$ using the ring $\Lambda$ in infinitely many variables, some explanation is called for.
	
	For $l(\lambda) \leq n$, i.e. $\lambda=(\lambda_1 \geq \cdots \geq \lambda_n \geq 0)$, let $s_{\lambda}(n) \in \Z[t_1, \ldots, t_n]$, and $s_{\lambda} \in \Lambda$. 
	Note that the identity,
	\begin{equation}\label{eq slein prod}
		s_{\lambda}(n) \cdot e_i(n)= \sum s_{\mu}(n),
	\end{equation}
	the first form of Pieri's rule as in Theorem~\ref{thm ext dp}, translates into:
	$$s_{\lambda} \cdot e_i= \sum s_{\mu}. $$ 
	
	More precisely,
	if 
	$$s_{\lambda}(n) \cdot s_{\mu}(n)=\sum L_{\lambda, \mu}^{\nu}(n) s_{\nu}(n),$$
	and 
	$$s_{\lambda} \cdot s_{\mu}=\sum L_{\lambda, \mu}^{\nu} s_{\nu},$$
	then $$L_{\lambda, \mu}^{\nu}=L_{\lambda, \mu}^{\nu}(n) \text{ if } l(\lambda) \leq n, l(\mu) \leq n, l(\nu) \leq n,  $$
	as follows from the specialization homomorphism $\phi_n: \Lambda \rightarrow \Z[t_1, \ldots, t_n]$ which has the property that
	% $\phi_n(s_{\lambda})=s_{\lambda}(n)$ if $l(\lambda) \leq n$, and $\phi_n(s_{\lambda})=0$ otherwise.
	$$\phi_n(s_{\lambda})=\begin{cases}
		s_{\lambda}(n) & \text{if } l(\lambda) \leq n,\\
		0 & \text{otherwise.}
	\end{cases} $$
	
	Thus we have,
	$$s_{\lambda} \cdot e_i= \sum s_{\mu}, $$ 
	and therefore,
	\begin{equation}\label{eq sh infty}
		s_{\lambda^t} \cdot h_i= \sum_{\mu^t} s_{\mu^t}.
	\end{equation}

	Once again, the specialization homomorphism allows us to deduce 
	\begin{equation}\label{eq sh finite}
		s_{\lambda^t}(n) \cdot h_i(n)= \sum_{\mu^t} s_{\mu^t}(n),
	\end{equation}
	where to be sure, the sum on the right hand side of the equalities in (\ref{eq sh infty}) and (\ref{eq sh finite}) need not be the same, more precisely the sum in the RHS of (\ref{eq sh finite}) consists of those terms $s_{\mu^t}$ in (\ref{eq sh infty}) of length $ l(\mu^t)\leq n$.
\end{proof}

In the literature, various proofs of Pieri's rules can be found. For example, see \cite[Theorem 6.1]{APschur} for a combinatorial proof using labeled Abaci.

The proof of Theorem~\ref{thm ext dp} when carried out for general reductive groups gives the following more general theorem, versions of which are available in the literature, for example see \cite[Corollary 3.4]{Shrawantensor}.

\begin{theorem}\label{thm pieri dp}
	Let $(G, B, T)$ be a triple of reductive group $G$, a Borel subgroup $B$, and a maximal torus $T \subseteq B$. Let $\rho$ be half the sum of the positive roots of $T$ in $B$, and $\lambda : T \to \C^{*}$, a dominant weight. Let $U$ be any irreducible representation of $G$ such that $\lambda + \rho + \mu$ is a dominant weight for all $\mu$,  characters of $T$ appearing in $U$. Then 
	\begin{equation*}
		\pilam \otimes U = \sum_{\mu} m(\mu) \Pi_{\lambda + \mu},
	\end{equation*}
	where $\mu : T \to \C^{*}$ runs over all characters of $T$ which appear in $U$ with multiplicity $m(\mu)$,  with the property that $\lambda + \mu$ is dominant. (The assumption in the theorem is only that $\lambda + \rho + \mu$ is  dominant, so for some $\mu : T \to \C^{*}$ appearing in $U$, $\lambda + \mu$ may not be dominant.)
\end{theorem}

\begin{corollary}
	If $U$ is a minuscule representation of $G$, then the hypothesis of Theorem~\ref{thm pieri dp} is automatically satisfied, hence
	$$
	\Pi_\lambda \otimes U \;=\; \sum_{\mu} \Pi_{\lambda+\mu},
	$$
	where $\mu$ runs over $W$-orbit of the highest weight of $U$ and each weight occurs with multiplicity $1$, restricted to those $\mu$ such that $\lambda+\mu$ is dominant.
\end{corollary}

\begin{proof}
	If $U$ is minuscule, its set of $T$-weights is a single $W$-orbit and every weight has multiplicity $1$. Moreover, for any simple coroot $\alpha^\vee$ and any weight $\mu$ of $U$, one has
	$$
	\langle \mu, \alpha^\vee\rangle \in \{ -1,0,1\}.
	$$
	Since $\lambda$ is dominant, $\langle \lambda,\alpha^\vee\rangle\ge 0$ for all simple $\alpha^\vee$, and therefore
	$$
	\langle \lambda+\rho+\mu,\alpha^\vee\rangle
	= \langle \lambda,\alpha^\vee\rangle + \langle \rho,\alpha^\vee\rangle + \langle \mu,\alpha^\vee\rangle
	\ge 0 + 1 + (-1) \;=\; 0.
	$$ 
\end{proof}

\begin{remark}
	 In Theorem~\ref{thm pieri dp}, the condition that $\lambda + \rho + \mu$ is a dominant weight for all $\mu$ is known to be sufficient. Based on Sagemath computations and examples, it appears that this condition is also necessary for the tensor product multiplicity to coincide exactly with the corresponding weight multiplicity. Although we have not found any counterexamples, a formal proof of this necessity remains open.
\end{remark}
An example of Theorem~\ref{thm pieri dp} for symplectic group as follows.
	\begin{example}
	Let $\lambda=\left\{ \lambda_1 > \lambda_2 > \cdots > \lambda_n > 0 \right \}$  be a regular sequence of integers and $\pilam$ be the corresponding irreducible representations of $\Sp{(2n, \C)}$.
	Consider the tensor product 
	$\pilam \otimes \Sym^2{(V)}$.
	 $\lambda + \rho + \mu$ is dominant for all weight $\mu$  in $\Sym^2{(V)}$, satisfies the hypothesis of Theorem~\ref{thm pieri dp}, hence
	$ \Pi_{\lambda + \mu} \subseteq \pilam \otimes \Sym^2{(V)}$,
	furthermore multiplicity of $\Pi_{\lambda + \mu}$ in the above tensor product is the multiplicity of the weight $\mu$ in $\Sym^2{(V)}$.
\end{example}

\section{Pieri for Strictly Dominant weights for Classical Groups}
The following theorem
is the Pieri Rule about the decomposition of $\pilam \otimes \Lambda^i V $ as a sum of irreducible representations of $G(V)$ for strictly dominant weights $\lambda$.

\begin{theorem}\label{thm pieri dp2}
	Let $V$ be a finite dimensional vector space over $\C$ with a non-degenerate quadratic or symplectic form and $G(V)$ the corresponding isometry group with determinant $=1$. Let $W$ be an $n$-dimensional maximal isotropic subspace of $V$ such that for another maximal isotropic subspace $W^* \subseteq V$, 
	$$V=\begin{cases}
		W \oplus W^* & \text{ if dim } V \text{ is even}, \\
		W \oplus W^* \oplus \C  & \text{ if dim } V \text{ is odd}.
	\end{cases} $$
	Let $\lambda=(\lambda_1 \geq \cdots \geq \lambda_n )$ be a dominant weight for $G(V)$. Assume that $\lambda$ is regular, i.e., $\lambda=(\lambda_1 > \cdots > \lambda_n >0 )$. Then the hypothesis of Theorem~\ref{thm pieri dp} is satisfied for  $\pilam \otimes \Lambda^i V $, and as a consequence we have the following conclusions.
	\begin{enumerate}
		\item [(i)] $G(V)=\Sp(V)$, i.e., $V=W \oplus W^{*}$ with a symplectic form. Then 
		\begin{equation*}
			\pilam \otimes \Lambda^i V = \sum_{\mu} m(\mu) \pimu,
		\end{equation*}
		where $\mu=(\mu_1 \geq \cdots \geq \mu_n \geq 0)$ such that there exists $\xi \supseteq \lambda$, $\xi \supseteq \mu$ (where containment $\xi \supseteq \lambda$, $\xi \supseteq \mu$  means just containment of Young diagrams), such that $\xi$ is obtained from $\lambda$ by adding at most one box from each row, and similarly $\mu$ is obtained by removing at most one box to each row of $\xi$ such that 
		$$ (|\xi|-|\lambda|) + (|\xi|-|\mu|)=i.$$
		The multiplicity $m(\mu)$ is the cardinality of $\xi$ such that $ (|\xi|-|\lambda|) + (|\xi|-|\mu|)=i$. 
		\item [(ii)] $G(V)=\SO(V)$ with $V=W \oplus W^{*}$. In this case, the recipe about  $\pilam \otimes \Lambda^i V$ is exactly as for $\Sp(V)$.
		\item [(iii)] $G(V)=\SO(V)= \SO(W \oplus W^{*} \oplus \C)$. In this case, the recipe about $\pilam \otimes \Lambda^i V$ is exactly as for the previous two cases except that instead of
		$$ (|\xi|-|\lambda|) + (|\xi|-|\mu|)=i,$$
		we have
		$$ (|\xi|-|\lambda|) + (|\xi|-|\mu|)=i \text{ or } i-1.$$
	\end{enumerate} 
\end{theorem}

\begin{proof}
	In case (i) and (ii), $V= W \oplus W^*$.
	Let $W=\{e_1, \ldots, e_n\}$, $W^{*}=\{f_1, \ldots, f_n\}$ such that $\langle e_i, f_j \rangle = \delta_{ij}$.
	
	Let $T$ be the maximal torus of $\G(V)$ defined by 
	\begin{equation*}
		te_i=t_ie_i, \quad tf_i=t_i^{-1}f_i, \quad \forall \ 1 \leq i \leq n. 
	\end{equation*}
	Since 
	\begin{equation*}
		\Lambda^i (V)=\Lambda^i (W \oplus W^*)= \sum_{j+k=i} \Lambda^j(W) \otimes \Lambda^k(W^*),
	\end{equation*}
	the characters of $T$ which appear in $\Lambda^i V$ are 
	\begin{equation}\label{eq t prod}
		t^\gamma=t_{\alpha_1} \cdots t_{\alpha_j} \cdot t_{\beta_1}^{-1} \cdots t_{\beta_k}^{-1},
	\end{equation}
	with $$ 1 \leq \alpha_1 < \alpha_2 < \cdots < \alpha_j \leq n, \quad 1 \leq \beta_1 < \beta_2 < \cdots < \beta_k \leq n. $$
	We have
	$$
	\rho=\begin{cases}
		ne_1 + (n-1)e_2 + \cdots + 2e_{n-1} + e_n & \text{ if } \quad G=\Sp(2n), \\
		(n-1)e_1 + (n-2)e_2 + \cdots + e_{n-1} & \text{ if } \quad G=\SO(2n).
	\end{cases}$$
	Therefore, writing $\rho= \rho_1 > \rho_2 > \cdots > \rho_n$, we have $\rho_i - \rho_{i+1}=1$ for all $i$, and as by hypothesis, $\lambda_i - \lambda_{i+1} \geq 1$ for all $i$ , we have: 
	\begin{equation}\label{eq lr diff}
		(\lambda_i + \rho_i) - (\lambda_{i+1} + \rho_{i+1}) \geq 2.
	\end{equation}
	
	From equation~(\ref{eq t prod}) and equation~(\ref{eq lr diff}), we find that
	$\lambda + \rho + \gamma$ are dominant for all weights $\gamma : T \to \C^*$ appearing in $\Lambda^i V= \Lambda^i (W \oplus W^*)$. 
	
	Therefore Theorem~\ref{thm pieri dp} applies, and we deduce this theorem from the characters of $T$ appearing in $\Lambda^i V$ through the equation~(\ref{eq t prod}). Putting $\lambda + \gamma = \mu$ in equation~\ref{eq t prod} implies 
	$$	t_{\alpha_1} \cdots t_{\alpha_j} \cdot t_{\beta_1}^{-1} \cdots t_{\beta_k}^{-1}=t^\gamma=t^{-\lambda + \mu}=t^{(-\lambda + \xi)} \cdot t^{\mu - \xi}=t^{(\xi-\lambda  )} \cdot t^{-(\xi-\mu)}.$$  
	More precisely,	 for
	$$t^{(\xi-\lambda  )} = t_{\alpha_1} \cdots t_{\alpha_j},  \quad t^{-(\xi-\mu)} = t_{\beta_1}^{-1} \cdots t_{\beta_k}^{-1},$$ 
	where Young diagram $\xi \supseteq \lambda$, $\xi \supseteq \mu$, we let $t^{\mu - \xi}$ and $t^{-(\lambda - \xi)}$ have the obvious meaning.

	For the case of  $\SO(V)= \SO(W \oplus W^{*} \oplus \C)$, 
	\begin{align*}
		\Lambda^i V&=\Lambda^i (W \oplus W^{*} \oplus \C) \\
		&= \Lambda^i (W \oplus W^{*}) \oplus \Lambda^{i-1}(W \oplus W^{*})\\
		&= \sum_{j+k=i \text{ or } i-1} \Lambda^jW \otimes \Lambda^kW^* ,
	\end{align*} 
	the rest of the conclusions are as in the case of $\Sp(V)$ and $\SO(W \oplus W^*)$ after we have noted that:
	$$\rho=\frac{2n-1}{2}e_1 + \frac{2n-3}{2}e_2 + \cdots + \frac{1}{2} e_n .$$ 
\end{proof}

\begin{remark}
	The version of Pieri rule for $\pilam \otimes \Sym^iV$ can be obtained from the Pieri rule for $\pilam \otimes \Lambda^i V$ using the duality involution $\omega : \Lambda \rightarrow \Lambda$ just as for $\GL(n, \C)$. This duality involution takes the basis $\chi_{\Sp}(\lambda)$ to $\chi_{\SO}(\lambda^t)$, allowing us to deduce the Pieri rule for  $\pilam \otimes \Sym^iV$. We refer to \cite{SO} for the statements and proofs.
\end{remark}

	\section{A Theorem of Kostant and a Generalization}
	The following theorem is due to Kostant, \cite{kostant}. The proof below is essentially his.
	
	\begin{theorem}\label{thm kostant initial}
		Let $(G, B, T)$ be a triple consisting of a reductive group $G$, a Borel subgroup $B$, and a maximal torus $T\subseteq B$ with $B=TN$. Assume that $B^{-}=TN^{-}$ is the opposite Borel subgroup. Let $\Pi_{\lambda}$ and $\Pi_{\mu}$ be two highest weight representations of $G$ associated to the dominant integral weights $\lambda$, $\mu$ $:T\to \C^{*}$. Then if a finite dimensional irreducible representation $\Pi_{\nu}$ of $G$ of highest weight $\nu : T \to \C^{*}$ appears in $\Pi_{\lambda} \otimes \Pi_{\mu}$, we must have 
		$$ \nu= \lambda + \tilde{\mu},$$
		where $\tilde{\mu}$ is a weight of $T$ in $\Pi_{\mu}$. The multiplicity of $\Pi_{\nu}$ in the
		tensor product $\Pi_{\lambda} \otimes \Pi_{\mu}$ is bounded by the multiplicity of  $\tilde{\mu}$ inside $\Pi_{\mu}$.
	\end{theorem}
	
	\begin{proof}
		Observe that 
		$$ \Pi_{\nu} \subseteq \Pi_{\lambda} \otimes \Pi_{\mu} \iff \text{ Hom}_G(\Pi_{\nu } \otimes \Pi_{\lambda }^{*}, \Pi_{\mu} ) \neq 0.$$
		Next observe that $\Pi_{\lambda }^{*}$ is the highest weight module of $G$ for the Borel subgroup $B^{-}=TN^{-}$ for the character $\lambda^{-1}: T \to \C^{*}$. We will treat representations of $G$ as representations of $\mathfrak{g}$ as well as representations of $\mathcal{U}(\mathfrak{g})$, the enveloping algebras of $\mathfrak{g}$, containing enveloping algebras of $\mathfrak{b}$, $\mathfrak{t}$, $\mathfrak{n}$, $\mathfrak{n}^{-}$, with the Poincare-Birkhoff-Witt theorem providing the decomposition
		$$\mathcal{U}(\mathfrak{g})=\mathcal{U}(\mathfrak{n}^{-}) \hspace*{1mm} \mathcal{U}(\mathfrak{t}) \hspace*{1mm} \mathcal{U}(\mathfrak{n}). $$ 
		
		Let $v_{\nu} \in \Pi_{\nu}$, and $v_{\lambda} \in \Pi_{\lambda}^{*}$ be the highest weight vectors for $B$ and $B^{-}$ respectively. As $\Pi_{\nu}$ and $\Pi_{\lambda}^{*}$ are irreducible.
		\begin{align*}
			\Pi_{\nu}&=\mathcal{U}(\mathfrak{g}) v_{\nu} =\mathcal{U}(\mathfrak{n}^{-}) v_{\nu},\\
			\Pi_{\lambda}^{*}&=\mathcal{U}(\mathfrak{g}) v_{\lambda} =\mathcal{U}(\mathfrak{n}^{+}) v_{\lambda}.
		\end{align*}
		It is easy to see therefore that 
		\begin{align*}
			\Pi_{\nu } \otimes \Pi_{\lambda }^{*} &= \left(\mathcal{U}(\mathfrak{n}^{-}) v_{\nu}\right) \otimes \left(\mathcal{U}(\mathfrak{n}^{+}) v_{\lambda}\right), \\
			&=\left(\mathcal{U}(\mathfrak{n}^{-}) \otimes \mathcal{U}(\mathfrak{n}^{+}) \right)\left(v_{\nu} \otimes v_{\lambda}\right), \\
			&= \mathcal{U}(\mathfrak{g}) \left(v_{\nu} \otimes v_{\lambda}\right).
		\end{align*}
		Therefore 
		$$\text{Hom}_{\mathcal{U}(\mathfrak{g})} \left( \Pi_{\nu } \otimes \Pi_{\lambda }^{*} , \Pi_{\mu} \right) \subseteq  \text{Hom}_{T} \left(v_{\nu} \otimes v_{\lambda} , \Pi_{\mu} \right).$$
		As $T$ acts by $\nu \cdot \lambda^{-1}$ on $v_{\nu} \otimes v_{\lambda}$, 
		$$\text{Hom}_{T} \left(v_{\nu} \otimes v_{\lambda} , \Pi_{\mu} \right) = \Pi_{\mu} (\nu \cdot  \lambda^{-1}).$$
		Thus $$\nu = \lambda + \text{ a weight in } \Pi_{\mu} ,$$
		completing the proof of the theorem.	 
	\end{proof}
	
	\begin{remark}\label{rem strong ckt}
	As stated in Kostant's Theorem~\ref{thm kostant initial}, if $ \Pi_{\nu} \subseteq \Pi_{\lambda} \otimes \Pi_{\mu} $, then 
	$\nu = \lambda + \text{a weight in } \Pi_{\mu}$. 
	A natural question is whether the converse holds: if 
	$\nu = \lambda + \text{a weight in } \Pi_{\mu}$, 
	is dominant, does it follow that 
	$ \Pi_{\nu} \subseteq \Pi_{\lambda} \otimes \Pi_{\mu} $? 
	Theorem~\ref{thm pieri dp} may be viewed as a strong form of this converse to 
	Kostant's Theorem~\ref{thm kostant initial}.
	
\end{remark}	

\begin{remark}\label{rem ckt in pieri}
	\begin{enumerate}
		\item We note that as the representation $\Lambda^iV$ of $\GL(V)$ is multiplicity free,
		the tensor product $\pilam \otimes \Lambda^iV$  is multiplicity free by Kostant's Theorem \ref{thm kostant initial}.
		The Pieri rule for general linear groups in the case of $\Lambda^iV$ (Theorem~\ref{thm ext dp})  is equivalent to saying that the converse of Kostant's Theorem \ref{thm kostant initial} holds.  However, Pieri rule for general linear groups in the case $\Sym^iV$ (Theorem~\ref{thm pieri gen dp}), although $\pilam \otimes \Sym^iV$ is multiplicity free, is not equivalent to the converse of Kostant's Theorem; in more detail, there are dominant weights $\nu$ of $\GL(V)$ of the form $ \nu= \lambda + \tilde{\mu}$,
		where $\tilde{\mu}$ is a weight  in $\Sym^iV$ but $\Pi_{\nu}$ does not appear in the
		tensor product $\Pi_{\lambda} \otimes \Sym^iV$. As an example, consider the tensor product decomposition $\Pi_{(2,1,1)} \otimes \Sym^2(\C^3)=\Pi_{(3,2,1)} \oplus \Pi_{(4,1,1)}$ as $\GL(3)$ modules. Weights of $\Sym^2(\C^3)$ are $(1, 1, 0)$, $(1, 0, 1)$, $(0, 1, 1)$, $(2, 0, 0)$, $(0, 2, 0)$, $(0, 0, 2)$. We have a dominant wight $\nu=\lambda + \tilde{\mu}=(2, 1, 1)+(0, 1, 1)=(2,2,2)$ for which the irreducible representation $\Pi_{(2,2,2)}$ does not appear in the tensor product $\Pi_{(2,1,1)} \otimes \Sym^2(\C^3)$.
		
		\item Computations using SageMath software suggest that the converse of Kostant's theorem is true in the context of Pieri rules for $\Lambda^iV $ for symplectic and orthogonal groups too. Specifically, $\Pi_{\nu} \subseteq \Pi_{\lambda} \otimes \Lambda^iV $ if and only if $\nu = \lambda + (\text{a weight in } \Lambda^iV  )$ is dominant, without  equality of multiplicities. 
	\end{enumerate}
	 
\end{remark}

	The above Theorem \ref{thm kostant initial} and proof works verbatim by replacing $B$ by a parabolic $P$, giving us the following theorem whose proof we will omit.
	
	\begin{theorem}\label{thm gbt parabolic}
		Let $(G, B, T)$ be a triple consisting of a reductive group $G$, a Borel subgroup $B$, and a maximal torus $T\subseteq B$. Let $P \supseteq B$ be a parabolic subgroup of $G$ with $P=MN$, $T \subseteq M$, its Levi decomposition. Let $P^{-}=MN^{-}$ be the opposite parabolic. Let $\Pi_{\lambda}$ and $\Pi_{\mu}$ be two finite dimensional highest weight representations of $G$ associated to the dominant integral weights $\lambda$, $\mu$ $:T\to \C^{*}$. Let $M_{\lambda}$ and $M_{\mu}$ be the corresponding finite dimensional highest weight representations of $M$. Then if a finite dimensional irreducible representation $\Pi_{\nu}$ of $G$ of highest weight $\nu: T \to \C^{*}$ appears in $\Pi_{\lambda} \otimes \Pi_{\mu}$, we must have 
		$$M_{\nu} \subseteq M_{\lambda} \otimes \left( \Pi_{\mu}\left.|\right._M \right),$$
		as $M$ modules. The multiplicity of $\Pi_{\nu}$ inside
		tensor product $\Pi_{\lambda} \otimes  \Pi_{\mu}$ as a $G$ module is bounded by the multiplicity of $M_{\nu}$ inside  $M_{\lambda} \otimes \left( \Pi_{\mu}\left.|\right._M \right)$ as $M$ modules.
	\end{theorem}

	We use Theorem~\ref{thm gbt parabolic} to make an observation on representations of classical groups 
	$$\Sp(V)=\Sp(W \oplus W^{*}), \text{ and } \SO(V)=\SO(W \oplus W^{*}) .$$
	
	\begin{theorem}\label{thm kostant parabolic}
		Let $\Pi_{\lambda}$ be an irreducible highest weight module of $G(V)=\Sp(V)=\Sp(W \oplus W^{*})$ or $G(V)=\SO(V)=\SO(W \oplus W^{*}) $, where $\lambda : T \to \C^{*}$ is a dominant integral weight. The group $G(V)$ has a maximal parabolic $P=MN$ with $M=\GL(W) \supseteq T$. Let $M_{\lambda}$ be the corresponding highest weight module of $\GL(W)$. Then for a dominant integral weight $\nu:T\to \C^{*}$, and an integer $0 \leq i \leq \dim V$, if we have 
		$$\Pi_{\nu} \subseteq \Pi_{\lambda} \otimes \Lambda^i(V),$$
		there are integers $(j,k)$ with $0 \leq j \leq  \dim V$,  $0 \leq k \leq \dim V$ such that 
		$$\text{Hom}_{\GL(W)} \left(M_{\nu}\otimes \Lambda^kW, M_{\lambda}\otimes \Lambda^j W  \right) \neq 0,$$
		where $M_{\lambda}$, $M_{\nu}$ are the highest weight representations of $\GL(W)$ with highest weights $\lambda, \nu  : T \to \C^{*}$. 
	\end{theorem}
	\begin{proof}
		By Theorem~\ref{thm gbt parabolic}, if 
		$$\Pi_{\nu} \subseteq \Pi_{\lambda} \otimes \Lambda^i(V),$$
		then 
		$$M_{\nu} \subseteq M_{\lambda} \otimes \left(\Lambda^iV\right)|_{\GL(W)}.$$
		But as $V=W \oplus W^{*}$, 
		$$\Lambda^iV=\sum_{j+k=i} \Lambda^jW \otimes \Lambda^kW^{*}, $$
		therefore if
		$$M_{\nu} \subseteq M_{\lambda} \otimes \left(\Lambda^iV\right)|_{\GL(W)},$$ then
		$$M_{\nu} \subseteq M_{\lambda} \otimes \Lambda^j W \otimes \Lambda^k W^{*},$$
		for some $j$ and $k$ with $j+k=i$.
		Equivalently,
		$$\text{Hom}_{\GL(W)} \left(M_{\nu} \otimes \Lambda^k W , \hspace*{1mm} M_{\lambda} \otimes \Lambda^j W  \right) \neq 0, $$
		proving the theorem.
	\end{proof}
	
	\begin{remark}\label{rem equivalent to CKT}
	By Pieri rule for $\GL(W)$, submodules of $M_{\nu} \otimes \Lambda^k W $ are $M_{\xi}$ such that the Young diagram of $\xi$ is obtained from the Young diagram of $\nu$ by adding $k$ boxes no two in the same row. 
	
	Similarly $M_{\xi} \subseteq M_{\lambda} \otimes \Lambda^j W $ if and only if the Young diagram of $\xi$ is obtained from the Young diagram of $\lambda$ by attaching $j$ boxes, no two in the same row. Thus if 
	$$\text{Hom}_{\GL(W)} \left(M_{\nu} \otimes \Lambda^k W , \hspace*{1mm} M_{\lambda} \otimes \Lambda^j W  \right) \neq 0, $$
	then
	$$|\xi| - |\nu|= k, \quad |\xi| - |\lambda|= j.$$ 
	Therefore we have obtained a necessary condition on submodules of $\Pi_{\lambda} \otimes \Lambda^i(V)$, which turns out to be sufficient for symplectic groups.
	% but not quite so for other classical groups. 
	
	The computations above imply that Pieri's rule about 
	$$\Pi_{\nu} \subseteq \Pi_{\lambda} \otimes \Lambda^iV$$
	is equivalent to the converse to extended Kostant's theorem, i.e., 
	$$\Pi_{\nu} \subseteq \Pi_{\lambda} \otimes \Lambda^iV \iff M_{\nu} \subseteq M_{\lambda} \otimes \left(\Lambda^iV \right)|_{\GL(W)}  $$ as $\GL(W)$ modules; more precisely
	$$\dim \Hom_{\Sp(2n)} \left(\Pi_{\nu}, \Pi_{\lambda} \otimes \Lambda^iV\right)= \dim \Hom_{\GL(W)} \left(M_{\nu},   M_{\lambda} \otimes \left(\Lambda^iV \right)|_{\GL(W)} \right).$$
\end{remark}

	\begin{remark}
		If $\lambda$, $\mu$ are highest weight of irreducible representations $\Pi_{\lambda} $, $\Pi_{\mu}$ of $\Sp(2n, \C)$, let 
		\begin{align*}
			\lambda^{+}&=\lambda +(1,1,\ldots, 1),\\
			\mu^{+}&=\mu   +(1,1,\ldots, 1).
		\end{align*}	
		Then it follows for the Pieri rule that 
		$$m\left(\Pi_{\mu}, \Pi_{\lambda} \otimes \Lambda^d \C^{2n} \right) = m\left(\Pi_{\mu^{+}}, \Pi_{\lambda^{+}} \otimes \Lambda^d \C^{2n} \right) .$$
		Corresponding statement is not true about symmetric powers.
	\end{remark}

	\begin{remark}
		The above proof reducing Pieri rule for classical groups from the Pieri rule for $\GL(n, \C)$ has been asserted only for $\Sp(2n, \C)$ and $\SO(2n, \C)$, but works just as well for $\SO(2n+1, \C)$ with the only difference that instead of $j+k=i$ for $\Sp(2n, \C)$ and $\SO(2n, \C)$, we have $j+k=i$ or $j+k=i-1$ as in Theorem~\ref{thm pieri dp2}(iii). Furthermore, the form of Pieri rule for $\Pi_{\lambda} \otimes \Sym^i(V)$ for classical groups reduces to $\GL(W)$ in exactly the same way. Since 
		$$\Sym^i(W \oplus W^{*})=\sum_{j+k=i} \Sym^j W \otimes \Sym^k (W^{*}),$$
		just as 
		$$\Lambda^i(W \oplus W^{*})=\sum_{j+k=i} \Lambda^jW \otimes \Lambda^kW^{*}. $$
	\end{remark}

	\section{Equivalence of Pieri and Branching for General Linear Groups}\label{ch equivalance}
	
	In this section, we establish the equivalence between branching from $\GL(n+1, \C)$ to $\GL(n, \C)$ and the Pieri rule for general linear groups. 
	%It is based on ideas around Howe correspondence and the seesaw duality, though we do not appeal to anything.
	
	\subsection{A Case of Howe Duality}
	In this subsection we analyze the natural action of $\GL(n, \C) \times \GL(m, \C)$, $n \leq m$, on the space of polynomial functions on the space of $n \times m$ matrices. This is the simplest case of what is called the Howe duality. This will be needed in the next subsection when we prove the equivalence of Pieri rule and the branching law from $\GL(n+1, \C)$ to $\GL(n, \C)$.
	\begin{lemma}
		Let $G$ be a reductive group operating on an affine variety $X$ with the ring of polynomial functions $\mathcal{S}(X)$. Suppose $B=TU$ is a Borel subgroup in $G$ which operates on $X$ with an open orbit passing through $x_0 \in X$, i.e., with $Bx_0$ an open subset of $X$. Suppose $T_0 \subset T$ is the stabilizer of $x_0 \in X$ in $T$. Then,
		\begin{enumerate}
			\item  The $G$-module $\mathcal{S}(X)$ is multiplicity free.
			\item  If an irreducible highest weight module $\pi_{\lambda}$, $\lambda : T \to \C^*$, is a submodule of $\mathcal{S}(X)$, then $\lambda|_{T_0}=1$. 
			\item If for $\lambda_1 : T \to \C^*$, $\lambda_2 : T \to \C^*$, the highest weight modules $\pi_{\lambda_1} \subseteq \mathcal{S}(X)$, $\pi_{\lambda_2} \subseteq \mathcal{S}(X)$, then $\pi_{\lambda_1 \lambda_2} \subseteq \mathcal{S}(X)$ also. (Here $\lambda_1 \lambda_2$ is the product of the characters $\lambda_1$ and $\lambda_2$).
		\end{enumerate} 
	\end{lemma}
	
	\begin{proof}
		By the theory of highest weights, $\mathrm{Hom}_G(\pi_{\lambda}, \mathcal{S}(X))$ is in bijective correspondence with functions $f$ on $X$ such that 
		$$ bf=\lambda(b) \cdot f.$$
		Since $B$ has an open orbit on $X$, this needs to be checked only on the open orbit, easily giving the three conclusions of the lemma.
	\end{proof}
	For further details on Borel dense orbit, see \cite[Proposition 1.3.3]{HoweIsrael} or \cite[Lemma 2.12]{Michelbrion}.
	
	\begin{corollary}\label{cor smnc}
		Let $G=\GL(n, \C) \times \GL(n, \C)$ operate on $M(n, \C)$ by $(g_1, g_2)X=g_1Xg_2^{-1}$, then 
		$$\mathcal{S}(M(n, \C))= \bigoplus_{\lambda } \pi_{\lambda} \otimes \pi_{\lambda}^{\vee},$$
		where  $\lambda=(\lambda_1 \geq \cdots \geq \lambda_n \geq 0)$ is a highest weight for $\GL(n, \C)$.
	\end{corollary}
	
	\begin{proof}
		This follows from the previous lemma on noting that $B \subseteq G=\GL(n, \C) \times \GL(n, \C)$ has an open orbit on $M(n, \C)$. More precisely, if
		\begin{equation*}
			B=B_1 \times B_1^0
		\end{equation*}
		where $B_1$ is the group of upper triangular matrices, and $B_1^0$ the group of lower triangular matrices, then $$B_1 \cdot \mathrm{Id} \cdot B_1^0 \text{ is open.}$$  Furthermore, in the notation of the previous lemma, 
		$T_0 \subset T \times T  $ consists of $T_0=\left \{(t,t) | t \in T\right\}$.
		Therefore $\chi= \chi_1 \times \chi_2$,
		\begin{equation*}
			\text{a character of } T \times T \text{ is trivial on }  T_0 \iff \chi_1=\chi_2^{-1}.
		\end{equation*}
		This is exactly the condition under which $\pi_{\chi_1} \otimes \pi_{\chi_2}$ where $\pi_{\chi_1}$ the highest weight module for $\GL(n, \C)$ for the Borel subgroup $B_1$ and $\pi_{\chi_2}$ the highest weight module for $\GL(n, \C)$ for the opposite Borel subgroup $B_1^0$, becomes $\pi_{\chi_1} \otimes \pi_{\chi_1}^{\vee}$.
		
		At this point, we have proved that
		$$\mathcal{S}(M(n, \C)) \subseteq \bigoplus_{\lambda  } \pi_{\lambda} \otimes \pi_{\lambda}^{\vee}.$$
		On the other hand $\pi_{\lambda} \otimes \pi_{\lambda}^{\vee}$ clearly sits inside $\mathcal{S}(M(n, \C))$ using matrix coefficients of the representation $\pi_{\lambda}$, completing the proof of the corollary.
	\end{proof}
	
	\begin{corollary}\label{cor Howe}
		For $n \leq m$, $G=\GL(n, \C) \times \GL(m, \C)$, operates on $X=\{n\times m \text{ matrices over } \C \}$. Then 
		$$\mathcal{S}(X)= \bigoplus_{\lambda } \pi_{\lambda} \otimes \pi_{\lambda}^{\vee},$$
		where $\lambda=(\lambda_1 \geq \cdots \geq \lambda_n \geq 0)$ giving a highest weight for $\GL(n, \C)$ as well as $\GL(m, \C)$ for all $m \geq n $.
	\end{corollary}
	
	\begin{proof}
		Once again, the action of $G$ on $X$ has the property that $X$ has an open $B$-orbit. In fact, if 
		$$B=B_1 \times B_2^0 \subseteq \GL(n, \C) \times \GL(m, \C),$$ where $B_1$ is the group of upper triangular matrices in $\GL(n, \C)$ and $B_2^0$ is the group of lower triangular matrices in $\GL(m, \C)$, then
		\[B \cdot J \text{ is open, \quad where }
		J= \begin{pNiceArray}{c|c}
			I_{n \times n} &  O
		\end{pNiceArray}.
		\]	
		Furthermore, the subgroup $T_0 \subseteq T_1 \times T_2$ which stabilizes $J$ is 
		$ T_0 = \{ (t_1, t_2)| t_1\in T_1$  arbitrary, and  $ t_2=t_1$  in the first  $n$ co-ordinates and no    condition in the last  $(m-n)$  co-ordinates\}.
		
		Therefore, characters $(\chi_1, \chi_2)$ of $T_1 \times T_2$ which are trivial on $T_0$ exactly give rise to the representations $\pi_{\lambda} \otimes \pi_{\lambda}^{\vee}$. This proves that 
		$$\mathcal{S}(X) \subseteq \bigoplus_{\lambda } \pi_{\lambda} \otimes \pi_{\lambda}^{\vee},$$
		where $\lambda=(\lambda_1 \geq \cdots \geq \lambda_n \geq 0)$.
		To prove the equality, consider the natural subspace $Y=M(n, \C) \subseteq X$, as
		\[
		Y= \begin{pNiceArray}{c|c}
			A & O
		\end{pNiceArray}, \quad A \in M(n, \C).
		\]	 
		
		This gives rise to a surjective map of $\GL(n, \C) \times \GL(n, \C)$-modules:
		$$ \mathcal{S}(X) \longrightarrow \mathcal{S}(Y) \longrightarrow 0,$$
		but by the previous corollary $$\mathcal{S}(Y)= \bigoplus_{\lambda  } \pi_{\lambda} \otimes \pi_{\lambda}^{\vee},$$
		therefore the inclusion \\
		$$\mathcal{S}(X) \subseteq \bigoplus_{\lambda } \pi_{\lambda} \otimes \pi_{\lambda}^{\vee},$$ must also be an equality.
	\end{proof}
	
	Corollary \ref{cor smnc}-\ref{cor Howe} are well known in the literature, see \cite{HoweIsrael} for further details.
	
	\subsection{Proof of Equivalence of Pieri and Branching}\label{subsec equiv pieri branch}
	Let $W$ be an $n$ dimensional vector space over $\C$. Let $V=W\oplus \C$, an $(n+1)$ dimensional space over $\C$. We have the following identity: 
	\begin{equation}\label{eq hom idty}	
		\mathrm{Hom}_{\C}(W,V)= \mathrm{Hom}_{\C}(W,W) \bigoplus \mathrm{Hom}_{\C}(W, \C). 
	\end{equation}

	These are vector spaces over $\C$ with natural action of $\GL(W) \times \GL(V)$ on the left of equation (\ref{eq hom idty}), and $\GL(W) \times \GL(W)$ on the right of equation (\ref{eq hom idty}) on the vector spaces with the second copy of $\GL(W)$ in $\GL(W) \times \GL(W)$ acting trivially on  $\mathrm{Hom}_{\C}(W, \C)$.
	
	Whenever a group operates on a vector space, $U$, it operates on the polynomial algebra $\mathcal{S}(U)$ on the vector space $U$. By Corollary~\ref{cor Howe}, $\mathcal{S}(\mathrm{Hom}
	(W, V))$ is a direct sum of the finite dimensional representations $\pi_{\mu} \otimes \pi_{\mu^+}^{\vee}$ where $\pi_{\mu}$ is any irreducible finite dimensional representation of $\GL(W)$, where $\mu=\{\mu_1 \geq \mu_2 \geq \cdots \mu_n \geq 0\}$ and $\pi_{\mu^+}^{\vee}$ is the dual of the corresponding representation of $\GL(V)=\GL(n+1, \C)$ where $\mu^+=(\mu_1, \mu_2, \ldots, \mu_n, 0)$ is a highest weight of $\GL(n+1, \C)$. Thus (by Corollary~\ref{cor Howe}), 
	\begin{equation}\label{eq shwv}
		\mathcal{S}(\mathrm{Hom}
		(W, V))= \sum_{\mu} \pi_{\mu} \otimes \pi_{\mu^+}^{\vee},
	\end{equation}
	as a representation of $\GL(W) \times \GL(V)$, where $\mu=\{\mu_1 \geq \mu_2 \geq \cdots \mu_n \geq 0\}$.
	On the other hand, by Corollary~\ref{cor smnc}, 
	$$\mathcal{S}(\mathrm{Hom}
	(W, W))= \sum_{\lambda} \pi_{\lambda} \otimes \pi_{\lambda}^{\vee},$$
	as a representation of $\GL(W) \times \GL(W)$, where $\lambda=(\lambda_1 \geq \cdots \geq \lambda_n \geq 0)$, and $$\mathcal{S}(\mathrm{Hom}
	(W, \C))= \mathcal{S}(W),$$
	as a $\GL(W) \times \GL(W)$ in which the second $\GL(W)$ acts trivially.
	Therefore, as a $\GL(W) \times \GL(W)$ module:
	\begin{align}
		\mathcal{S}\left(\mathrm{Hom}
		(W, W) \bigoplus \mathrm{Hom}
		(W, \C) \right) & \cong \left(\sum_{\lambda} \pi_{\lambda} \otimes \pi_{\lambda}^{\vee}\right) \bigotimes \left(\mathcal{S}(W) \otimes \C \right) \nonumber \\
		& \cong \sum_{\lambda} \left(\pi_{\lambda} \otimes \mathcal{S}(W)  \right) \otimes \pi_{\lambda}^{\vee} 
		\label{pisw}.
	\end{align}
	
	Equating equations (\ref{eq shwv} ) and (\ref{pisw}) as $\GL(W) \times \GL(W)$ modules, we have:
	
	\begin{equation}\label{equality big}
		\sum_{\mu} \pi_{\mu} \otimes \pi_{\mu^+}^{\vee} \cong \sum_{\lambda} \left(\pi_{\lambda} \otimes \mathcal{S}(W)  \right) \otimes \pi_{\lambda}^{\vee}
	\end{equation}
	where $\pi_{\mu}$ is a representation of $\GL(W)$ and $\pi_{\mu^{+}}$ is a representation of $\GL(V)$ with $\mu^+=(\mu_1, \mu_2, \ldots, \mu_n, 0)$, thus adding a zero at the end of $\mu$.

	Instead of considering the action of $\GL(W) \times \GL(V)$ on the space $X$ of $n \times (n+1)$ matrices, given by $g_1Xg_2^{-1}$, we will consider the action of $\GL(W) \times \GL(V)$ on the space of $n \times (n+1)$ matrices given by $g_1X (\leftidx{^t}{g_2})$, which removes taking duals in equation~(\ref{equality big}), which for the new action of $\GL(W) \times \GL(V)$ becomes equation~(\ref{equality after}) as the equation~(\ref{equality big}) without duals: 
	
	\begin{equation}\label{equality after}
		\sum_{\mu } \pi_{\mu} \otimes \pi_{\mu^{+}} = \sum_{\lambda} \left(\pi_{\lambda} \otimes S(W)  \right) \otimes \pi_{\lambda}.
	\end{equation}
	Decompose 
	$$\pi_{\mu^+}|_{\GL(n, \C)}= \sum \pi_{\lambda}, $$ and rewrite equation~(\ref{equality after}) as:
	\begin{equation}\label{eqn: equivalance equation}
		\sum_{\pi_{\lambda} \subseteq \pi_{\mu^+}|_{\GL(n, \C)}} \pi_{\mu} \otimes \pi_{\lambda} =  \sum_{\pi_{\mu} \subseteq \pi_{\lambda}\otimes S(W)} \pi_{\mu} \otimes \pi_{\lambda}.
	\end{equation}
	The decomposition of RHS is the Pieri Rule according to which 
	$\pi_{\mu} \otimes \pi_{\lambda}$, a representation of $\GL(n, \C) \times \GL(n, \C)$, appears in RHS $\iff$ one can obtain $\mu$ from $\lambda$ by adding a skew Young diagram to $\lambda$ consisting of a horizontal strip. 
	
	On the other hand,  $\pi_{\mu} \otimes \pi_{\lambda}$, a representation of $\GL(n, \C) \times \GL(n, \C)$,
	appears in LHS $\iff$ $\pi_{\lambda} \subseteq \pi_{\mu^+}|_{\GL(n, \C)}$
	$\iff$ one can obtain $\mu$ from $\lambda$
	by adding a skew Young diagram to $\lambda$ consisting of a horizontal strip.
	
	This proves equivalence of the Pieri rule for $\GL(n, \C)$ with that of branching law for $\pi_{\lambda^+}$, a representation of $\GL(n+1, \C)$. Since any irreducible representation of $\GL(n+1, \C)$ can be assumed to be of the from  $\pi_{\lambda^+}$, by twisting by the determinant character, proving equivalence of Pieri rule with that of branching rule from  $\GL(n+1, \C)$ to $\GL(n, \C)$.

	It may be noted that the equality 
	$$\sum_{\pi_{\lambda} \subseteq \pi_{\mu^+}|_{\GL(n, \C)}} \pi_{\mu} \otimes \pi_{\lambda} =  \sum_{\pi_{\mu} \subseteq \pi_{\lambda}\otimes S(W)} \pi_{\mu} \otimes \pi_{\lambda}$$
	seems not amenable to be written term-by-term. It is a certain reciprocity which can be expressed as 
	\begin{equation}\label{eqn: reciprocity}
		\mathrm{Hom}_{\GL(W)}\left(\pi_{\lambda} \otimes S(W), \pi_{\mu} \right) \cong  \mathrm{Hom}_{\GL(W)}(\pi_{\mu^+}, \pi_{\lambda}) 
	\end{equation}
	where we recall that $\pi_{\mu^+}$ is a representation of $\GL(V)$. This can also be written as the following seesaw diagram 
	
	\[ \xymatrix{
		\pi_{\mu^+} \ar@{-}[d] \ar@{-}[rd] & \pi_{\lambda} \otimes S(W) \ar@{-}[d] \ar@{-}[ld] \\%
		\pi_{\lambda} & \pi_{\mu}
	}
	\]
	in which $\pi_{\mu^+}$ is a representation of $\GL(V)$ restricted to $\GL(W)$ but all the other terms are representations of $\GL(W)$, for the following diagram of groups:
	\[ \xymatrix{
		\GL(n+1) \ar@{-}[d] \ar@{-}[rd] & \GL(n) \otimes \GL(n) \ar@{-}[d] \ar@{-}[ld] \\%
		\GL(n) \otimes \GL(1) & \GL(n) 
	}
	\]
	
	$$  \mathrm{Hom}_{\GL(W)}(\pi_{\mu^+}, \pi_{\lambda})\cong \mathrm{Hom}_{\GL(W)}\left(\pi_{\lambda} \otimes \Sym^k, \pi_{\mu} \right)   .$$
	
\vspace{6mm}
 \noindent \textbf{Acknowledgment} 
 The author thanks Prof. Dipendra Prasad for carefully reading the manuscript and pointing out several corrections. Part of this work was carried out as a component of the author's doctoral thesis at IIT Bombay under the supervision of Prof. Dipendra Prasad, and another part was completed at the Chennai Mathematical Institute. The author gratefully acknowledges partial financial support from the Infosys Foundation during the course of this research.

\end{document}